\documentclass[letterpaper, 10 pt, conference]{ieeeconf}  

\IEEEoverridecommandlockouts                              
\usepackage{graphics} 
\usepackage{epsfig} 
\usepackage{amsfonts,amsmath}
\usepackage{amssymb}

\newtheorem{definition}{Definition}

\newtheorem{proposition}{Proposition}
\newtheorem{lemma}{Lemma}
\newtheorem{theorem}{Theorem}

\newtheorem{algorithm}{Algorithm}

\newtheorem{remark}{Remark}

\title{\LARGE \bf
A Time-Periodic Lyapunov Approach for Motion Planning  of \\
Controllable Driftless Systems on $\mbox{SU}(n)$}

\author{H. B. Silveira, P. S. Pereira da Silva and P. Rouchon
\thanks{The first author was fully supported by CAPES. The third author was partially supported by CNPq. The second and third authors were partially
supported by CAPES/COFECUB and ``Agence Nationale de la Recherche'' (ANR), Projet Blanc CQUID number 06-3-13957.}
\thanks{H. B. Silveira and P. S. Pereira da Silva are with Laboratory of Automation and Control, Department of Telecommunications and Control Engineering, University of S\~ao Paulo, Brazil
        {\tt\small hectorbessa@yahoo.com.br} and {\tt\small paulo@lac.usp.br}}
\thanks{P. Rouchon is with Mines ParisTech, Centre Automatique et Syst\`emes, Math\'ematiques et Syst\`emes, France \newline
        {\tt\small pierre.rouchon@mines-paristech.fr}}
}

\begin{document}

\bibliographystyle{plain}

\maketitle
\thispagestyle{empty}
\pagestyle{empty}

\begin{abstract}

For a right-invariant and controllable driftless system on
$\mbox{SU}(n)$, we consider a time-periodic reference trajectory along which
the linearized control system generates $\mathfrak{su}(n)$: such
trajectories always exist and constitute the basic ingredient of
Coron's Return Method. The open-loop controls that we propose, which
rely on a left-invariant tracking error dynamics and on a
fidelity-like Lyapunov function, are determined from a finite
number of left-translations of the tracking error and they assure
global asymptotic convergence towards the periodic reference trajectory. The
role of these translations is to avoid being trapped in the critical
region of this Lyapunov-like function. The convergence proof relies
on a periodic version of LaSalle's invariance principle and the control values are
determined by numerical integration of the dynamics of the system.
Simulations illustrate the obtained
controls for $n=4$ and the generation of the C--NOT quantum gate.
\end{abstract}

\section{INTRODUCTION}

Consider the right-invariant driftless system
\begin{equation}\label{cqs}
    \dot{X}=\sum_{k=1}^{m}u_k H_k X, \hspace{12pt} X(0)=I,
\end{equation}
where $X \in M^n$ is the state, $M^n$ is the Banach space of square
$n \times n$ matrices with complex entries endowed with the
Euclidean norm, $H = \{ H_1,\dots, H_m \} \subset \mathfrak{su}(n)$,
$u_k \in \mathbb{R}$ are the controls, and $I$ is the identity
matrix of $M^n$. The \emph{periodic motion planning problem} for this
system is formulated as follows. Given a \emph{goal state}
$X_{\infty} \in \mbox{SU}(n)$ and $T > 0$, find a smooth periodic
\emph{reference trajectory} $X_r$:~$\mathbb{R}_+ \rightarrow \mbox{SU}(n)$
of period $T$, with $X_r(0)=X_{\infty}$, and determine continuous
open-loop controls $u_k$:~$\mathbb{R}_+ \rightarrow \mathbb{R}$, for
$1 \leq k \leq m$, in a manner that the tracking error between the
trajectory
$X$:~$\mathbb{R}_+ \rightarrow \mbox{SU}(n)$ of (\ref{cqs}) and $X_r$
converges to zero as $t \rightarrow \infty$, that is, $\lim_{t
\rightarrow \infty} [X(t) - X_r(t)] = 0$.

We remark that there
is no loss of generality in assuming that $X(0)=I$ in (\ref{cqs}). Indeed, since system (\ref{cqs}) is right-invariant, if
$(X(t), (u_1(t), \dots, u_m(t)))$, for $t \in \mathbb{R}_+$, is a
solution of (\ref{cqs}) with $X(0)=I$, then $(X(t) X_0,
(u_1(t), \dots, u_m(t)))$, for $t \in \mathbb{R}_+$, is a
solution of (\ref{cqs}) with initial condition $X(0)=X_0 \in
\mbox{SU}(n)$. Therefore, if the periodic motion planning problem
has been solved for system (\ref{cqs}) with $X(0)=I$, it is straightforward to show that it will also be
solved for (\ref{cqs}) with $X(0)=X_0 \in \mbox{SU}(n)$.


The main result of this paper is the determination of a solution for the periodic motion planning problem. This is established by
Theorem~\ref{sat} in Section~2, whose only
assumption is that system (\ref{cqs}) regular, in the sense of Definition~\ref{rsd} in Section~2. The results of Coron's Return Method show
that such condition is always met in case the system is controllable on $\mbox{SU}(n)$ (see Remark~\ref{crm} in Section~2).
Loosely speaking, by finding an appropriate reference trajectory $X_r$,
using the time-dependent change of
coordinates $Z=Z(X,t)=X^\dag X_r(t)$, which corresponds to the tracking error on the group $\mbox{SU}(n)$,
and defining an adequate ``feedback'', we determine an algorithm that obtains, in a finite number of steps, continuous open-loops controls $u_k$, for every $1 \leq k \leq m$, which assure that
the tracking error $X-X_r$ converges to zero as $t \rightarrow \infty$. This algorithm relies on Lyapunov-like convergence results inspired in the
periodic version of LaSalle's invariance principle presented in
\cite{Vidyasagar93}, and in the \emph{ad-condition} stabilization method of
\cite{JurQui78}.
In a certain sense, we have used the real part of the trace of the left-invariant tracking error
$Z$ as a Lyapunov-like function, that is, $V(Z)=\Re(tr(Z))$.
In the case of quantum systems,
$V$ can then be seen as a fidelity-like Lyapunov function.

The problem of steering a quantum system from a given initial state to an arbitrary final state,
which can be regarded as a particular case of the periodic motion planning problem here formulated,
has recently been treated in \cite{SilRou08b} using a flatness-based approach and in the book \cite{Ale08}
(see also the references therein), where many quantum control techniques used in the literature are grouped together and explained in detail,
such as Lyapunov-based methods, optimal control and decompositions of $\mbox{SU}(n)$.
Our Lyapunov-like approach has no restrictions on the goal state $X_\infty \in \mbox{SU}(n)$ and on $n$, as long as system (\ref{cqs}) is regular.

The layout of the paper is as follows. Section~2 is entirely dedicated to the proof of Theorem~\ref{sat} mentioned above.
Simulations illustrate in Section~3 the generation of the Controlled-NOT
(C--NOT) gate for a quantum system with $n=4$. Appendix presents the proof of the important convergence result of Theorem~\ref{dct} in Section~2.

\section{Main Result}

Based on (\ref{cqs}), we define the \emph{reference system}
\begin{equation}\label{rqs}
    \dot{X_r}=\sum_{k=1}^{m}u^r_{k} H_kX_r, \hspace{12pt} X_r(0)=X_{\infty} \in \mbox{SU}(n),
\end{equation}
where $X_r \in M^n$ and the smooth time functions $u^r_{k}$:~$\mathbb{R}
\rightarrow \mathbb{R}$ are still to be specified.

\begin{definition}\label{rsd}
System (\ref{cqs}) is said to be \emph{regular} when, given $T > 0$, there
exist smooth periodic functions $u^T_{k}$:~$\mathbb{R} \rightarrow
\mathbb{R}$ of period $T$, for all $1 \leq k \leq m$, such that the
solution $X^T_r$:~$\mathbb{R}_+ \rightarrow \mbox{SU}(n)$ of
(\ref{rqs}), with $X^T_r(0)=I$ and $u^r_k = u^T_k$, is also periodic
of period $T$ and satisfies
\begin{align}\label{rss}
    \mbox{span} \{ B^j_k(0), 1 \leq k \leq m, j \in \mathbb{N} \} = \mathfrak{su}(n),
\end{align}
where $\mathbb{N}$ is the set of natural numbers (including zero),
$A(t)=\sum_{k=1}^{m} u^T_{k}(t) H_k \in \mathfrak{su}(n)$, $B^0_k(t) =
H_k X^T_r(t)$, $B^{j+1}_k(t) = - A(t)B^j_k(t) + \dot{B}^{j}_k(t)$, $1
\leq k \leq m$, $j \in \mathbb{N}$, $t \in \mathbb{R}$.
\end{definition}

\begin{remark}\label{rss-p}
Note that $A$:~$\mathbb{R} \rightarrow \mathfrak{su}(n)$, $B^j_k$,
$\dot{B}^j_k$:~$\mathbb{R} \rightarrow M^n$ are smooth and also have
period $T$, for every $1 \leq k \leq m$, $j \in \mathbb{N}$. Hence,
they are bounded mappings.
\end{remark}

\begin{remark}\label{crm}
Note that the linearized control system of (\ref{rqs}) (or of (\ref{cqs}))
along the trajectory $(X_r^T,(u_1^T,\dots,u_m^T))$ is given by
$\dot{X}_r^\ell = A(t) X_r^\ell + \sum_{k=1}^m w_k B^0_k(t)$, $w_k
\in \mathbb{R}$. Based on Coron's Return
Method (see \cite{Cor94}, \cite{Cor07}), it can be shown that (\ref{cqs}) is regular in case $\mbox{Lie}(H) = \mathfrak{su}(n)$.
We recall that (\ref{cqs}) is controllable on $\mbox{SU}(n)$ if and only if
$\mbox{Lie}(H) = \mathfrak{su}(n)$ \cite{AlbAle03}.
\end{remark}

For simplicity, we shall assume throughout this paper that
system (\ref{cqs}) is regular, that $T > 0$ has been fixed and that
the functions $u^r_{k}$ in (\ref{rqs}) were specified accordingly,
that is, $u^r_{k}=u^T_{k}$, for $1 \leq k \leq m$. Moreover, we
also assume that the goal state $X_\infty \in \mbox{SU}(n)$ is
fixed. Define $X_r$:~$\mathbb{R} \rightarrow \mbox{SU}(n)$ as $X_r =
X^T_r X_\infty$. Note that $X_r$ is the solution of (\ref{rqs}) with
$X_r(0) = X_\infty$ and that $X_r$ also has period $T$. It will be
shown afterwards that $X_r$ can indeed be used as a reference trajectory.
We also adopt the following notations. The imaginary unit of
$\mathbb{C}$ is denoted by $\imath$ and if $z \in \mathbb{C}$, then
$\Re(z)$ is its real part and $\Im(z)$ its imaginary part.

It is straightforward to verify from (\ref{cqs}) and (\ref{rqs})
that the time-dependent change of coordinates
\begin{equation*}
    Z=Z(t,X)=X^{\dag}X_r(t), \hspace{12pt} \mbox{ for all } (t,X) \in \mathbb{R} \times M^n,
\end{equation*}
along with the time-varying control shift
\begin{equation*}
    v_k \triangleq u_k^r(t) - u_k = u_k^T(t) - u_k, \hspace{12pt} \mbox{ for all } t \in \mathbb{R}, \; 1 \leq k \leq m,
\end{equation*}
determine the left-invariant ``closed-loop system''
\begin{equation}\label{clqs}
    \dot{Z} = Z X^\dag_r(t) \sum_{k=1}^{m} v_k H_k X_r(t), \hspace{8pt} Z(0)=X_\infty \in \mbox{SU}(n),
\end{equation}
for all $(t,Z) \in \mathbb{R} \times M^n$. If we can find continuous functions $v_k$:~$\mathbb{R}_+ \rightarrow
\mathbb{R}$, for each $1 \leq k \leq m$, such that
\begin{equation}\label{clsc}
    \lim_{t \rightarrow \infty} Z(t) = \lim_{t \rightarrow \infty} X^{\dag}(t) X_r(t) = I,
\end{equation}
where $Z$:~$\mathbb{R}_+ \rightarrow \mbox{SU}(n)$ is the solution of system (\ref{clqs}) and $X$:~$\mathbb{R}_+ \rightarrow \mbox{SU}(n)$ is the solution of system (\ref{cqs}) with the continuous open-loop controls
\begin{equation*}
    u_k(t) = u^T_{k}(t) - v_k(t), \hspace{12pt} \mbox{ for all } t \in \mathbb{R}_+, \; 1 \leq k \leq m,
\end{equation*}
it is then clear that
\begin{equation}\label{tec}
    \lim_{t \rightarrow \infty} [X(t) - X_r(t)] = 0,
\end{equation}
thus solving the periodic motion planning problem.
%

Let $V$:~$M^n \rightarrow \mathbb{R}$ be defined by
\begin{equation}\label{llf}
    V(X) = \Re(tr(X)), \hspace{12pt} \mbox{ for all } X \in M^n,
\end{equation}
and consider the \emph{auxiliar system}
\begin{equation}\label{as}
    \dot{W} = W X^\dag_r(t)\sum_{k=1}^{m} f_k a_k(t,W) H_k X_r(t),
\end{equation}
where $(t,W) \in \mathbb{R} \times M^n$, $f_k \neq 0$ is a fixed
real number, $1 \leq k \leq m$, and
\begin{equation}\label{fd}
    a_k(t,W) = f_k \mathrm{V}(W X^{\dag}_r(t) H_k X_r(t)). 
\end{equation}
Notice that the ``closed-loop'' system (\ref{clqs}) with ``feedbacks'' $v_k=f_k a_k(t,Z)$ is nothing but the auxiliar system (\ref{as})--(\ref{fd}). Note also that $V$ in (\ref{llf}) is linear and that, for $X \in
\mbox{SU}(n)$, we have $-n \leq V(X) \leq n$ and $V(X) =
n$ if and only if $X=I$. Furthermore, by construction, $\dot{\mathrm{V}}(t,W)=\sum_{k=1}^{m} a_k(t,W)^2 \geq 0$, for all $(t,W)
\in \mathbb{R} \times M^n$.

In what follows, we shall show how the next theorem, which is a
Lyapunov-like convergence result for the auxiliar system with Lyapunov-like function $V(W) = \Re(tr(W))$, and whose
proof is deferred to Appendix, determines continuous functions $v_k$:~$\mathbb{R}_+ \rightarrow
\mathbb{R}$, for $1 \leq k \leq m$, such that (\ref{clsc}) is satisfied for the
``closed-loop'' system (\ref{clqs}). We remark that the properties of $V$ stated above are essential in the proof.
Our approach to solve the periodic motion planning problem is then summarized in Theorem~\ref{sat}.

\begin{theorem}\label{dct}
Consider the set
\begin{align*}
    G = \{ & x \in \mathbb{R} : x = \sum_{i=1}^{n} \Re(\lambda_i), \mbox{ for some } \lambda_i \in \mathbb{C} \mbox{ such}\\
    & \mbox{that } | \lambda_i | = 1, \prod_{i=1}^{n} \lambda_i = 1, \Im(\lambda_1) = \dots = \Im(\lambda_n) \}.
\end{align*}
Then, $G$ is a finite set, $n \in G$ and $n=\max(G)$. Furthermore,
letting
$\delta$ be the maximal element of the set $G \setminus \{n\}$,
we have that, for all $q=(t_0, W_{t_0}) \in \mathbb{R} \times
\mbox{SU}(n)$,
\begin{equation*}
    V(W_{t_0}) > \delta \Rightarrow \lim_{t \rightarrow \infty} W_q(t) = I,
\end{equation*}
where $W_q$:~$\mathbb{R} \rightarrow \mbox{SU}(n)$ is the solution
of (\ref{as})--(\ref{fd}) with initial condition $W_q(t_0)=W_{t_0}$.
\end{theorem}




Suppose that $V(X_\infty) > \delta$. In Theorem~\ref{dct}, we choose $q=(0,X_\infty) \in \mathbb{R} \times \mbox{SU}(n)$. Therefore,
$\lim_{t \rightarrow \infty} W_q(t) = I$. Hence, the smooth ``feedbacks'' $v_k$:~$\mathbb{R}_+ \rightarrow
\mathbb{R}$ defined as
\begin{equation*}\label{fdd}
    v_k(t) \triangleq f_k a_k(t,Z(t)) = f_k^2 \mathrm{V}(X^\dag(t) H_k X_r(t)), 
\end{equation*}
for $t \in \mathbb{R}_+$, $1 \leq k \leq m$, assure that $Z(t)=W_q(t)$, for $t \in \mathbb{R}_+$. Indeed, compare (\ref{clqs}) with (\ref{as})--(\ref{fd}).
Thus, (\ref{clsc}) holds.

Now, assume that $V(X_\infty) \leq \delta$.
For this case, based on continuity arguments, we determine an adequate (continuous) path from $X_\infty$ to $I$ which, in a certain sense, reduces the problem to the situation where $V(X_\infty) > \delta$.
In order to achieve this, the main idea is to find a path $\overline{Z}$:~$[0, 1] \rightarrow \mbox{SU}(n)$, with $\overline{Z}(0)=X_\infty$ and
$\overline{Z}(1)=I$, and obtain $0=\theta_0 < \theta_1 < \dots < \theta_{N-1} < \theta_N=1$, such that
$V(\overline{Z}(\theta_{\ell+1})^\dag \overline{Z}(\theta_{\ell})) > \delta$, for all $0 \leq \ell \leq N-1$. It thus follows from Theorem~1 that,
for $1 \leq \ell \leq N$, $\lim_{t \rightarrow \infty} \overline{Z}(\theta_{\ell}) W_{\ell}(t) = \overline{Z}(\theta_{\ell})$, where
$W_\ell$:~$\mathbb{R} \rightarrow \mbox{SU}(n)$ is the
solution of (\ref{as})--(\ref{fd}) with initial condition
$W_{\ell}(T_\ell)=\overline{Z}(\theta_{\ell})^\dag \overline{Z}(\theta_{\ell-1}) \in \mbox{SU}(n)$, where $0 = T_1 < \dots < T_{N+1}$ are such that
$W_{\ell}(T_{\ell+1}) \thickapprox I$. Loosely speaking, we then ``glue'' together
the left-translations $\overline{Z}(\theta_{1})W_1, \dots, \overline{Z}(\theta_{N})W_N$ in an appropriate manner in order
to define a continuous solution $(Z(t), (v_1(t), \dots, v_m(t)))$, for $t \in \mathbb{R}_+$, of system (\ref{clqs}) that satisfies (\ref{clsc}).
We remark that, for every $1 \leq \ell \leq N$, it is as if
we were in the case $V(X_\infty) > \delta$. In the sequel, we formalise these arguments in detail and determine an algorithm which obtains, in $N$ steps, continuous functions $v_k$:~$\mathbb{R}_+ \rightarrow
\mathbb{R}$, for $1 \leq k \leq m$, such that (\ref{clsc}) holds.

It is a standard result that any $X_\infty \in \mbox{SU}(n)$ can be
written as
$X_\infty = M^\dag
\mbox{diag}(\exp{\imath \lambda_1},\dots,\exp{\imath \lambda_n}) M$,
where $M$ is a unitary matrix, $\lambda_1, \dots, \lambda_n \in
\mathbb{R}$ and $\sum^n_{i=1} \lambda_i = 0$. Consider the
path $\overline{Z}$:~$[0, 1] \rightarrow \mbox{SU}(n)$ from $X_\infty$ to $I$
defined by
\begin{align*}
    \overline{Z}(\theta)=M^\dag \mbox{diag}(\exp{\imath \lambda_1 (1-\theta) },\dots,\exp{\imath \lambda_n (1-\theta)}) M, 
\end{align*}
for all $\theta \in [0,1]$.
Let $a, b \in [0,1]$. Hence,
$\overline{Z}(b)^\dag \overline{Z}(a)=M^\dag \mbox{diag}(\exp{\imath
\lambda_1 (b-a)},\dots,\exp{\imath \lambda_n (b-a)}) M$ and
therefore $V(\overline{Z}(b)^\dag \overline{Z}(a))=\sum_{j=1}^n
\cos(\lambda_j (b-a))$. Since the function $\gamma$:~$[0,1]
\rightarrow \mathbb{R}$ defined by $\gamma(\theta)=\sum_{j=1}^n
\cos(\lambda_j \theta)$, for all $\theta \in [0,1]$, is continuous
with $\gamma(0)=n$, there exists $\nu > 0$ such that
$\gamma(\theta) > \delta$ in case $|\theta| < \nu$, for all $\theta \in [0,1]$
(indeed, choose $\epsilon = n - \delta > 0$). Hence,
$V(\overline{Z}(b)^\dag \overline{Z}(a)) > \delta$ whenever $|b-a| < \nu$, for all $a,b \in [0,1]$,
and there exists a non-zero $\eta \in \mathbb{N}$ such that, for all
$N \geq \eta$,
\begin{equation}\label{ic}
    V(\overline{Z}_{\ell+1}^\dag \overline{Z}_\ell)=\sum_{j=1}^n \cos(\lambda_j \Delta) > \delta, 
\end{equation}
for all $0 \leq \ell \leq N-1$, where $\overline{Z}_{\ell}=\overline{Z}(\theta_{\ell})$, $\theta_\ell=\ell \Delta$,
for every $0 \leq \ell \leq N$, with $\Delta=1/N$. Note that $\overline{Z}_0=
\overline{Z}(0)=X_\infty$ and
$\overline{Z}_N=\overline{Z}(1)=I$.
Let $N \geq \eta$ and consider the continuous function
$\beta$:~$M^n \times M^n \rightarrow \mathbb{R}$ defined by
$\beta(X,Y)=V(Y^\dag X)$, for all $(X,Y) \in M^n \times M^n$. Since
$\mbox{SU}(n) \times \mbox{SU}(n)$ is compact, $\beta|(\mbox{SU}(n)
\times \mbox{SU}(n))$ is uniformly continuous. Therefore, by (\ref{ic}), there exists $\mu > 0$ such that, for all $X \in
\mbox{SU}(n)$ and $0 \leq \ell \leq N-1$, we have
\begin{equation}\label{nic}
    \| X - \overline{Z}_\ell \| < \mu \Rightarrow V(\overline{Z}_{\ell+1}^\dag X) > \delta 
\end{equation}
(indeed, choose $\epsilon = \sum_{j=1}^n \cos(\lambda_j \Delta) -
\delta > 0$ and consider the sup norm on $M^n \times M^n$). The aforementioned algorithm is described below. Recall that $\overline{Z}_0=X_\infty$ and
$\overline{Z}_N=I$.

\begin{algorithm}
Let $X_\infty \in \mbox{SU}(n)$. Choose any non-zero $N \in \mathbb{N}$ in a manner that (\ref{ic}) holds. Define $T_1 = 0$ and $W_0(T_1)=I$.
For every $1 \leq \ell \leq N-1$,
choose a real number $T_{\ell + 1} > T_{\ell}$ such that $V(\overline{Z}_{\ell+1}^\dag \overline{Z}_{\ell}
W_{\ell}(T_{\ell+1})) > \delta$, where $W_\ell$:~$\mathbb{R} \rightarrow \mbox{SU}(n)$ is the
solution of the auxiliar system (\ref{as})--(\ref{fd}) with initial condition
$W_{\ell}(T_{\ell})=\overline{Z}_{\ell}^\dag \overline{Z}_{\ell-1}W_{\ell-1}(T_{\ell}) \in \mbox{SU}(n)$. Define
\begin{align*}
    Z(t) &= \overline{Z}_{\ell} W_{\ell}(t) \in \mbox{SU}(n), \, \hspace{57pt} \mbox{for } t \in [T_{\ell}, T_{\ell + 1}), \\
    v_k(t) &= f_k a_k(t,W_{\ell}(t)) \in \mathbb{R}, \, 1 \leq k \leq m, \, \mbox{for } t \in [T_{\ell}, T_{\ell + 1}).
\end{align*}
If $T_N > T_{N-1}$ has been chosen as above,  define
\begin{align*}
    Z(t) &= W_{N}(t) \in \mbox{SU}(n), &\mbox{ for } t \geq T_{N}, \\
    v_k(t) &= f_k a_k(t,W_{N}(t)) \in \mathbb{R}, \hspace{4pt} 1 \leq k \leq m, &\mbox{ for } t \geq T_{N},
\end{align*}
where $W_N$:~$\mathbb{R} \rightarrow \mbox{SU}(n)$ is the
solution of (\ref{as})--(\ref{fd}) with initial condition
$W_{N}(T_{N})=\overline{Z}_{N-1}W_{N-1}(T_{N}) \in \mbox{SU}(n)$. \hspace{26pt} $\blacksquare$
\end{algorithm}

Some remarks are in order. First of all, from the reasoning preceding Algorithm~1, we know
that there always exists some non-zero $N \in \mathbb{N}$ such that (\ref{ic}) is true. Furthermore, Theorem~\ref{dct} and property (\ref{nic}) assure that
$T_{\ell + 1} > T_{\ell}$ can always be chosen as required in the algorithm, for every $1 \leq \ell \leq N-1$. It is also clear
that $(Z(t), (v_1(t), \dots, v_m(t)))$, for $t \in \mathbb{R}_+$, determined by the algorithm is a continuous solution of the ``closed-loop'' system (\ref{clqs}).
Indeed, compare (\ref{clqs}) with (\ref{as})--(\ref{fd}).
Finally, since $V(\overline{Z}_{N}^\dag \overline{Z}_{N-1}
W_{N-1}(T_{N}))=V(\overline{Z}_{N-1}W_{N-1}(T_{N})) > \delta$, Theorem~\ref{dct} implies that $\lim_{t \rightarrow \infty} W_N(t)=I$. However, $Z(t)=W_N(t)$, for $t \geq T_N$. Therefore,
the continuous functions $v_k$:~$\mathbb{R}_+ \rightarrow
\mathbb{R}$ determined by Algorithm~1 are such that (\ref{clsc}) is satisfied. We have thus shown our main result:

\begin{theorem}\label{sat}
Assume that system (\ref{cqs}) is regular, in the sense of Definition~\ref{rsd}. Given
$X_{\infty} \in \mbox{SU}(n)$, $T > 0$ and ``feedback gains'' $f_k^2 > 0$, consider $X_r^T$:~$\mathbb{R} \rightarrow \mbox{SU}(n)$ and
$u_k^T$:~$\mathbb{R} \rightarrow \mathbb{R}$ as in Definition~\ref{rsd}, for $1 \leq k \leq m$. Define $X_r=X_r^T X_\infty$. Then, there exist continuous
open-loop controls $u_k$:~$\mathbb{R}_+ \rightarrow \nolinebreak \mathbb{R}$, for $1 \leq k \leq m$, such that (\ref{tec}) is satisfied.
In other words, the periodic motion planning problem always has a solution when (\ref{cqs}) is regular. More precisely, if
$V(X_\infty) > \delta$, where $\delta$ is as in Theorem~\ref{dct}, then the smooth open-loop controls $u_k(t) = u_k^T(t) - f_k^2 \mathrm{V}(X^\dag(t) H_k X_r(t))$, obtained by numerical integration,
for $t \in \mathbb{R}_+$, $1 \leq k \leq m$, assure that (\ref{tec}) holds.
Otherwise, in case $V(X_\infty) \leq \delta$, then by following Algorithm~1 we determine continuous functions $v_k$:~$\mathbb{R}_+ \rightarrow \mathbb{R}$, for $1 \leq k \leq m$, such
that the corresponding continuous open-loop controls $u_k(t) = u_k^T(t) - v_k(t)$, for $t \in \mathbb{R}_+$, assure that (\ref{tec}) is satisfied.
\end{theorem}


\section{QUANTUM MECHANICAL EXAMPLE}

After some approximations, an appropriate change of coordinates,
scalings and simplifications, a controlled quantum system consisting
of two coupled spin-$\frac{1}{2}$ particles with Heisenberg
interaction and driven by an external electromagnetic field, can be
modeled as \cite{Ale08}
\begin{equation}\label{tcsps}
    \dot{Y} = (D + D_x u_x + D_y u_y + D_z u_z) Y, \hspace{12pt} Y(0)=I,
\end{equation}
where $Y \in M^4$ ($n=4$), the controls $u_x$, $u_y$, $u_z \in
\mathbb{R}$ are the $x$, $y$, $z$ components of the electromagnetic
field, respectively, $D = \mbox{diag}(3\imath, -\imath, -\imath,
-\imath)$, $D_x = H^R_{14} - 3 H^R_{23}$, $D_y = H^R_{13} + 3
H^R_{24}$, $D_z = H^R_{12} - 3 H^R_{34} \in \mathfrak{su}(4)$, and
$H^R_{ij}=(h^{R,ij}_{k\ell})$, $H^I_{ij}=(h^{I,ij}_{k\ell}) \in
\mathfrak{su}(4)$ are the matrices with entries
\begin{align*}
        & h^{R,ij}_{ij} = 1, \hspace{4pt} h^{R,ij}_{ji} = -1, \hspace{4pt} h^{R,ij}_{k\ell}=0, \hspace{12pt} \mbox{ for } k, \ell \neq i, j,  \medskip \\
        & h^{I,ij}_{ij} = h^{I,ij}_{ji} = \imath, \hspace{31pt} h^{I,ij}_{k\ell}=0, \hspace{14pt} \mbox{ for } k, \ell \neq i, j,
\end{align*}
respectively, for all $1 \leq i < j \leq n$.

Now, in order to remove the drift term $DY$ in (\ref{tcsps}), we
define, as usual, the time-dependent change of coordinates
$X=\Phi(t,Y)=e^{- D t} Y$, for all $(t,Y) \in \mathbb{R} \times M^4$.
In these coordinates, (\ref{tcsps}) is described as\footnote{In
quantum mechanics, this description is usually called the
interaction picture or interaction representation.}
\begin{equation}\label{ipd}
    \dot{X} = (C_x u_x + C_y u_y + C_z u_z)X,  \hspace{12pt} X(0)=I,
\end{equation}
where
\begin{equation*}
    \begin{array}{l}
    C_x = e^{-Dt} D_x e^{Dt} = \left(
            \begin{array}{rrrl}
               0 & 0 &  0 & e^{-\imath 4 t} \\
               0 & 0 & -3 & 0 \\
               0 & 3 &  0 & 0 \\
              -e^{\imath 4 t} & 0 &  0 & 0 \\
            \end{array}
          \right), \medskip \\
    C_y = e^{-Dt} D_y e^{Dt} = \left(
            \begin{array}{rrlr}
               0 &  0 & e^{- \imath 4 t} & 0 \\
               0 &  0 & 0 & 3 \\
              -e^{\imath 4 t} &  0 & 0 & 0 \\
               0 & -3 & 0 & 0 \\
            \end{array}
          \right), \medskip \\
    C_z = e^{-Dt} D_z e^{Dt} = \left(
            \begin{array}{rlrr}
               0 & e^{- \imath 4 t} & 0 &  0 \\
              -e^{\imath 4 t} & 0 & 0 &  0 \\
               0 & 0 & 0 & -3 \\
               0 & 0 & 3 &  0 \\
            \end{array}
          \right),
    \end{array}
\end{equation*}
for all $t \in \mathbb{R}$. We choose the real controls $u_x$,
$u_y$, $u_z$ as
\begin{equation}\label{rc}
    \begin{array}{l}
        u_x = (u_1 + \imath u_2) e^{\imath 4 t} + (u_1 - \imath u_2) e^{-\imath 4 t}, \medskip \\
        u_y = (u_3 + \imath u_4) e^{\imath 4 t} + (u_3 - \imath u_4) e^{-\imath 4 t}, \medskip \\
        u_z = (u_5 + \imath u_6) e^{\imath 4 t} + (u_5 - \imath u_6) e^{-\imath 4 t},
    \end{array}
\end{equation}
respectively, for all $t \in \mathbb{R}$, where $u_1, \dots, u_6 \in \mathbb{R}$
are the new controls. By applying the rotating wave
approximation (RWA) (see e.g. \cite{Rou08eng},
\cite{Ale08}, \cite{HarRai06}) to system
(\ref{ipd})--(\ref{rc}), which consists in considering only the
terms that are time-independent and in disregarding all the oscillating
ones, we obtain the following time-independent driftless system
\begin{equation}\label{rwad}
    \dot{X} = (u_1 H_{14}^R + u_2 H_{14}^I + u_3 H_{13}^R + u_4 H_{13}^I + u_5 H_{12}^R + u_6 H_{12}^I )X,
\end{equation}
with initial condition $X(0)=I$. It is straightforward to verify
that $\mbox{Lie}(\{ H_{14}^R, H_{14}^I, H_{13}^R, H_{13}^I,
H_{12}^R, H_{12}^I \})=\mathfrak{su}(4)$, i.e. the system is
controllable on $\mbox{SU}(4)$. Hence, Coron's Return Method implies that the system
is regular (see Remark~\ref{crm}) and therefore Theorem~\ref{sat} can be applied.
We choose $T=1$ and as goal state the C--NOT (Controlled-Not)
gate
\begin{equation*}
    X_\infty =
    \displaystyle
    \left(
    \begin{array}{rrrr}
        1 & 0 & 0 & 0 \\
        0 & 1 & 0 & 0 \\
        0 & 0 & 0 & -1 \\
        0 & 0 & 1 & 0 \\
    \end{array}
    \right) \in \mbox{SU}(4),
\end{equation*}
which is one of the universal gates and has great importance in
quantum information theory \cite{NieChu00}, \cite{HarRai06}. It is easy to see from the proof of
Theorem~\ref{dct} that $G=\{-4, 0, 4\}$ with $\delta=0$. Since $V(X_\infty) = 2 > 0$, Theorem~\ref{sat} implies that the smooth open-loop controls
$u_k(t) = u_k^1(t) - f_k^2 \mathrm{V}(X^\dag(t) H_k X_r(t)) = u_k^1(t) - v_k(t)$, for $t \in \mathbb{R}_+$, $1 \leq k \leq 6$, obtained by numerical integration,
assure that $\lim_{t \rightarrow \infty} [X(t) - X_r(t)] = 0$, for any ``feedback gains'' $f_k^2 > 0$.
Here, $u_k^1 = u_k^T$ with $T=1$, and $H_1=H_{14}^R$, $H_2=H_{14}^I$, $H_3=H_{13}^R$, $H_4=H_{13}^I$, $H_5=H_{12}^R$, $H_6=H_{12}^I$.
However, the periodic functions $u_1^1,
\dots, u_6^1$ are not known explicitly. Coron's Return
Method only establishes their existence. Fortunately, for system (\ref{rwad}), symbolic computation software packages have shown
that if we define them as $u_k^1(t) = \sum_{\ell=1}^{n_f} a_{k\ell} \sin(2 \pi \ell t)$, for $t \in \mathbb{R}$, $1 \leq k \leq 6$,
with $n_f > 1$ and where $a_{k\ell} \in \mathbb{R}$ are randomly chosen
from the uniform distribution on the interval $[-a, a]$ with
``sufficiently large'' $a > 0$, then it is ``very likely''
that $\mbox{dim}(\mbox{span} \{ B^j_k(0), 1 \leq k \leq 6, 0 \leq j \leq 6 \})$ = 15, that is,
(\ref{rss}) holds (recall that $\mbox{dim}(\mathfrak{su}(4))=15$). And, when (\ref{rss}) is true, it
follows that $\lim_{t \rightarrow \infty} [X(t) - X_r(t)]$, where $X_r = X^1_r X_\infty$.
We remark that since $u_k^1$ is an odd periodic function with period $T=1$, the solution $X^1_r$:~$\mathbb{R} \rightarrow \mbox{SU}(n)$ in Definition~\ref{rsd} is also periodic with period $T=1$.
Note that $a$ and $n_f$ determine the ``excitation level'' of $u_k^1$. For $f_k=1$, computer simulations have suggested that as $a$ and $n_f$ get larger,
the faster the convergence of the tracking error $X - X_r$ to zero (assuming that $\mbox{dim}(\mbox{span}\{ B^k_j(0) \}) = 15$, of course).


The obtained simulation results are now presented for $f_k=1$, $a=n_f=5$ and $a_{k\ell}$ having as values the corresponding
entries of the matrix $\overline{A}=(a_{k\ell})$ below
\begin{equation*}
    \overline{A} =
    \left(
    \begin{array}{rrrrr}
        -2.00 & -1.39 &  4.66 &  4.31 &  1.80 \\
        -0.31 & -1.54 &  0.92 & -3.20 & -2.18 \\
        -4.69 & -0.31 &  1.75 &  3.94 & -1.11 \\
        -2.79 &  0.77 &  4.09 &  2.34 &  3.46 \\
         2.19 &  0.60 & -0.27 &  0.43 & -3.75 \\
        -0.18 & -4.44 & -1.38 & -4.58 &  2.59 \\
    \end{array}
    \right).
\end{equation*}
With these choices, we have indeed verified that $\mbox{dim}(\mbox{span}\{ B^k_j(0) \}) = 15$.
Figure~1 exhibits the convergence of $\| X - X_r \|$ to
zero (Euclidean norm on $M^4$). We see that the norm of the tracking error is non-increasing. In
Figure~2, the controls $u_1, u_2$ (top) and the ``feedbacks'' $v_1,
v_2$ (bottom) on the time interval $[0, 10]$ are shown. Notice that
$v_k$ is relatively small in comparison with the
control $u_k$, for $k=1,2$. Therefore, the control $u_k$ is
relatively close to $u_k^1$ as defined above, for $k=1,2$. In order
not to overwhelm the presentation, we have chosen not to exhibit $u_k, v_k$,
for $3 \leq k \leq 6$. They have, however, a similar behavior and a
similar order of magnitude as for $k=1,2$.

\begin{figure}[htbp]
    \scalebox{0.25}{
    \includegraphics{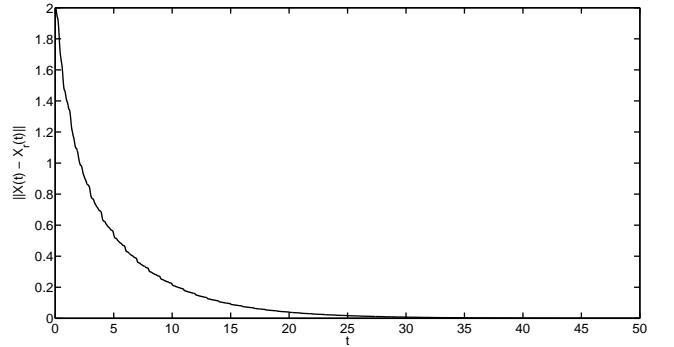}}
    \caption{Convergence of the norm of the tracking error to zero.}
\end{figure}

\begin{figure}[htbp]
    \hspace{-22pt}
    \scalebox{0.30}{
    \includegraphics{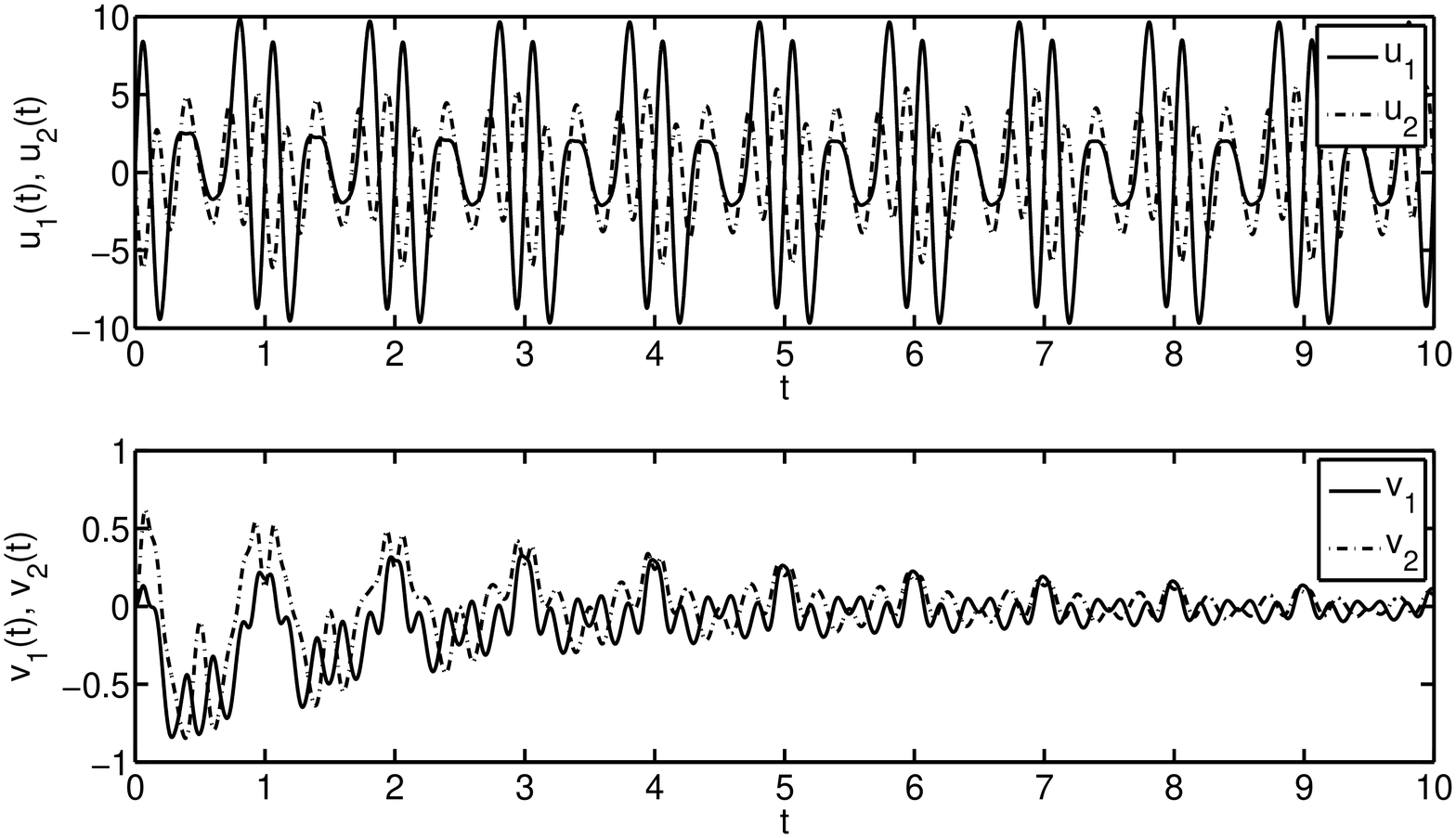}}
    \caption{Controls $u_1, u_2$ and ``feedbacks'' $v_1, v_2$ on the interval $[0,10]$.}
\end{figure}

\section{CONCLUDING REMARKS}

In the solution here presented for the periodic motion planning problem,
the only needed assumption is that system (\ref{cqs}) is
regular, which requires that the periodic functions $u_k^T$ satisfying (\ref{rss}) are
explicitly known. Nevertheless, this will hardly be the case in general.
For this reason, currently under investigation is
the explicit determination of $u_k^r$ in (\ref{rqs}) in a manner
that Theorem~\ref{dct} still holds under assumptions other than the
regularity of system (\ref{cqs}).

\section{ACKNOWLEDGMENTS}

The authors would like to thank Jean-Michel Coron and Mazyar
Mirrahimi for valuable discussions and suggestions.

\appendix

\section{Proof of Lemma~\ref{edcl-p}}

In order to prove Theorem~\ref{dct}, we need first a few
intermediate definitions and results. For simplicity,
we consider throughout this section that $q=(t_0,W_{t_0}) \in \mathbb{R} \times \mbox{SU}(n)$ is fixed and that
$W_q$:~$\mathbb{R} \rightarrow \mbox{SU}(n)$ denotes the solution of the auxiliar system
(\ref{as})--(\ref{fd}) with initial condition $W_q(t_0)=W_{t_0}$.

\begin{definition}\cite{Vidyasagar93}
A point
$\overline{W} \in M^n$ is called a \emph{limit point} of $W_q$ if
there exists a real sequence $\{ t_m \}$ such that $\lim_{m
\rightarrow \infty} t_m = \infty$ and $\lim_{m \rightarrow \infty}
W_q(t_m) = \overline{W}$. The set of all limit points of $W_q$ is
called the \emph{limit set} of $W_q$ and is denoted by $\Omega(W_q)$.
\end{definition}

\begin{remark}\label{lsp}
Since $\mbox{SU}(n)$ is a compact subset of $M^n$, it is clear that $\Omega(W_q)$
is a non-empty subset of $\mbox{SU}(n)$.
\end{remark}


\begin{proposition}\cite{Vidyasagar93}\label{cls}
$\lim_{t \rightarrow \infty} d(W_q(t),\Omega(W_q)) = 0$.
\end{proposition}

The next $2$ lemmas are essential in the proof of the important
convergence result of Theorem~\ref{gct} given below, which was inspired in the
periodic version of LaSalle's invariance principle presented in
\cite{Vidyasagar93} and in the \emph{ad-condition} stabilization method of
\cite{JurQui78}.

\begin{lemma}\label{cips}
Let $W$:~$\mathbb{R} \rightarrow M^n$ be a continuously
differentiable mapping such that $\lim_{t \rightarrow \infty}
\dot{W}(t) = 0$. Suppose that $\{ t_m \}$ is a real sequence such
that $\lim_{m \rightarrow \infty} t_m = \infty$ and $\lim_{m
\rightarrow \infty} W(t_m) = \overline{W}$. Then, for every
$\epsilon \in \mathbb{R}$, we have that $\lim_{m \rightarrow \infty}
W(t_m + \epsilon) = \overline{W}$.
\end{lemma}

\begin{proof}
Let $\epsilon \geq 0$ and $m \in \mathbb{N}$. We have that $W(t_m +
\epsilon) - W(t_m) = \int_{t_m}^{t_m + \epsilon} \dot{W}(t) \, dt.$
Thus, the inequality $\| W(t_m + \epsilon) - W(t_m) \| \leq |
\epsilon | \sup_{t \in [t_m , t_m + \epsilon ]}  \| \dot{W}(t) \|$
holds. The assumptions then imply that $\lim_{m \rightarrow \infty}
W(t_m + \epsilon) = \overline{W}$. For $\epsilon < 0$, we can
proceed in an analogous manner.
\end{proof}

\begin{lemma}\label{gl}
Consider that
$\overline{W} \in \Omega(W_q)$, $j \in \mathbb{N}$ and let $1
\leq k \leq m$. Assume that $\lim_{t \rightarrow \infty}
\dot{W_q}(t) = 0$ and that $\lim_{t \rightarrow \infty}
\mathrm{V}(W_q(t) X^{\dag}_r(t) B^j_k(t) X_\infty) = 0$, where
$B^j_k$ is as in (\ref{rss}). Then, $\mathrm{V}(\overline{W}
X^{\dag}_{\infty} B^j_k(0) X_\infty) = 0$.
\end{lemma}

\begin{proof}
Let $\overline{W} \in \Omega(W_q)$. By definition, there
exists a real sequence $\{ t_m \}$ such that $\lim_{m \rightarrow
\infty} t_m = \infty$ and $\lim_{m \rightarrow \infty} W_q(t_m) =
\overline{W}$. Now, for each $m \in \mathbb{N}$, there exists
$\ell_m \in \mathbb{Z}$ such that $s_m = t_m - \ell_m T \in [0,T)$,
where $T> 0$ is the period of $X_r$ and of $B^j_k$ (see
Remark~\ref{rss-p}). Since $[0,T]$ is compact, there exists a
subsequence $\{ s_{m_i} \}$ in which $\lim_{i \rightarrow \infty}
s_{m_i} = \theta \in [0,T]$. Let $\{ t_{m_i} \}$ be the
corresponding subsequence of $\{ t_{m} \}$. Define the sequences $\{
t^*_{m_i} \}$ and $\{ s^*_{m_i} \}$ as $t^*_{m_i} = t_{m_i} -
\theta$ and $s^*_{m_i} = s_{m_i} - \theta$, respectively. We have
that $\lim_{t \rightarrow \infty} \dot{W}_q(t) = 0$ as well as
$\lim_{t \rightarrow \infty} \mathrm{V}(W_q(t) X^{\dag}_r(t)
B^j_k(t) X_\infty) = 0$ (assumptions). Therefore, by definition,
$\lim_{i \rightarrow \infty} s^*_{m_i} = 0$, and Lemma~\ref{cips}
gives that $\lim_{i \rightarrow \infty} W_q(t^*_{m_i}) = \lim_{i
\rightarrow \infty} W_q(t_{m_i} - \theta) = \overline{W}$.
Hence,
the continuity and periodicity of $X_r$ and of $B^j_k$ imply that
$\lim_{i \rightarrow \infty} \mathrm{V}(W_q(t^*_{m_i})
X^{\dag}_r(t^*_{m_i}) B^j_k(t^*_{m_i}) X_\infty) = \lim_{i
\rightarrow \infty} \mathrm{V}(W_q(t^*_{m_i}) X^{\dag}_r(s^*_{m_i})
B^j_k(s^*_{m_i}) X_\infty) = \mathrm{V}(\overline{W} X^{\dag}_\infty
B^j_k(0) X_\infty)=0$.
\end{proof}

\begin{theorem}\label{gct}
Consider the subset $E = \{  W \in \mbox{SU}(n) : \mathrm{V}(W
X^{\dag}_{\infty} B^j_k(0) X_\infty) = 0$, for all $j \in
\mathbb{N}, 1 \leq k \leq m \}$,
where $B^j_k$ is as in (\ref{rss}).
Then, $\lim_{t \rightarrow
\infty} d(W_q(t),E) = 0$ and $E$ is non-empty.
\end{theorem}

\begin{proof}
Due to Proposition~\ref{cls}, it suffices to prove that the non-empty
limit set $\Omega(W_q)$ of the solution $W_q$ is contained
in the set $E$. We remark that since $\mathrm{V}$:~$M^n \rightarrow
\mathbb{R}$ is a continuous linear function, there exists $c > 0$
such that $| \mathrm{V}(X) | \leq c \| X \|$, for all $X \in M^n$.
Furthermore, it follows from (\ref{rqs}), (\ref{as})--(\ref{fd}),
Remark~\ref{rss-p} and the compactness of $\mbox{SU}(n)$ that each
of the mappings $X_r$, $X^\dag_r$, $W_q$, $B^j_k$, $\dot{X}_r$,
$\dot{X}^\dag_r$, $\dot{W}_q$, $\dot{B}^j_k$ is bounded, for every $j
\in \mathbb{N}$, $1 \leq k \leq m$.

Consider the functions $\alpha$:~$\mathbb{R} \rightarrow \mathbb{R}$,
$b^j_k$:~$\mathbb{R} \times M^n \rightarrow \mathbb{R}$,
$\beta^j_k$:~$\mathbb{R} \rightarrow \mathbb{R}$ defined
respectively as
\begin{equation*}
    \begin{array}{ll}
        \alpha(t) = \mathrm{V}(W_q(t)), & \mbox{for all } t \in \mathbb{R}, \\
        b^j_k(t,W) = \mathrm{V}(W X^{\dag}_r(t) B^j_k(t) X_\infty), & (t,W) \in \mathbb{R} \times M^n,\\
        \beta^j_k(t) = b^j_k(t,W_q(t)), & \mbox{for all } t \in \mathbb{R},
    \end{array}
\end{equation*}
for $j \in \mathbb{N}$, $1 \leq k \leq m$.
We will prove by induction that
\begin{equation}\label{cz-g}
    \lim_{t \rightarrow \infty} \beta^j_k(t)  = \lim_{t \rightarrow \infty} \mathrm{V}(W_q(t) X^{\dag}_r(t) B^j_k(t) X_\infty) = 0,
\end{equation}
for  $j \in \mathbb{N}$, $1 \leq k \leq m$. From (\ref{as})--(\ref{fd}) and the definition of $b^0_k$, we have that $\dot{\mathrm{V}}(t,W)=
\sum_{k=1}^{m} f_k^2 b^0_k(t,W)^2 \geq 0$ and
$\ddot{\mathrm{V}}(t,W) = 2 \sum_{k=1}^{m} f_k^2 \dot{b}^0_k(t,W) b^0_k(t,W)$,
where $\dot{b}^0_k(t,W) = \mathrm{V}(W X^{\dag}_r(t)
\sum_{\ell=1}^{m} f^2_\ell b^0_\ell(t,W) H_\ell B^0_k(t) X_\infty)
+ b^1_k(t,W)$
and $f_1,
\dots, f_m \in \mathbb{R}$ are non-zero, for $(t,W) \in \mathbb{R}
\times M^n$. Since $\dot{\mathrm{V}}$ is a non-negative function, we
conclude that $\alpha$ is a non-decreasing function bounded from
above such that $\ddot{\alpha}$ is bounded. Hence, $\lim_{t
\rightarrow \infty} \alpha(t)$ exists and is finite.
This relation along with Barbalat's Lemma (see e.g. \cite{Slotine91}) give that $\lim_{t \rightarrow \infty}
\dot{\alpha}(t) = \sum_{k=1}^{m} f_k^2 b^0_k(t,W_q(t))^2 = 0$. Thus,
$\lim_{t \rightarrow \infty} \beta^0_k(t) = \lim_{t \rightarrow
\infty} \mathrm{V}(W_q(t) X^{\dag}_r(t) B^0_k(t) X_\infty) = 0$, for
each $1 \leq k \leq m$, from which (\ref{as})--(\ref{fd}) imply that
\begin{equation}\label{cz-gct}
    \lim_{t \rightarrow \infty} \dot{W}_q(t) = 0.
\end{equation}
Now, consider the induction hypothesis
\begin{equation}\label{cz-g-ih}
    \lim_{t \rightarrow \infty} \beta^j_k(t) = \lim_{t \rightarrow \infty} \mathrm{V}(W_q(t) X^{\dag}_r(t) B^j_k(t) X_\infty) = 0,
\end{equation}
for some $j \in \mathbb{N}$ and all $1 \leq k \leq m$. We have that
$\dot{b}^j_k(t,W) = \mathrm{V}(W X^{\dag}_r(t) \sum_{\ell=1}^{m}
f^2_\ell b^0_\ell(t,W) H_\ell B^j_k(t) X_\infty) + b^{j+1}_k(t,W)$,
for all $1 \leq k \leq m$, $(t,W) \in \mathbb{R} \times M^n$.
Straightforward computations show that $\ddot{\beta}^j_k$ is
bounded because
$\ddot{\beta}^j_k(t)=\ddot{b}^j_k(t,W_q(t))$, for all $1 \leq k \leq
m$, $t \in \mathbb{R}$. Hence, (\ref{cz-g-ih}) and Barbalat's Lemma
imply that $\lim_{t \rightarrow \infty} \beta^{j+1}_k(t) = \lim_{t
\rightarrow \infty} \mathrm{V}(W_q(t) X^{\dag}_r(t) B^{j+1}_k(t)
X_\infty) = 0$, for $1 \leq k \leq m$. We have thus proved that
(\ref{cz-g}) is true. At this moment, it is simple to prove
that $\Omega(W_q) \subset E$. Indeed, assume that $\overline{W} \in \Omega(W_q) \subset \mbox{SU}(n)$.
Then, (\ref{cz-g}), (\ref{cz-gct}) and Lemma~\ref{gl} imply that
$\mathrm{V}(\overline{W} X^{\dag}_\infty B^j_k(0) X_\infty) = 0$, for
each $j \in \mathbb{N}$, $1 \leq k \leq m$.
\end{proof}

\begin{lemma}\label{edcl-p}
Consider the subset $F = \{ W \in \mbox{SU}(n) : V(W) = \sum_{i=1}^{n} \Re(\lambda_i), \mbox{ for some } \lambda_i \in \mathbb{C}
\mbox{ such that } | \lambda_i | = 1, \prod_{i=1}^{n} \lambda_i = 1, \Im(\lambda_1) = \dots = \Im(\lambda_n) \}$.
Then, $I \in F$ and $\lim_{t \rightarrow \infty} d(W_q(t),F) = 0$.
\end{lemma}
\begin{proof}
According to Theorem~\ref{gct}, it suffices to show the inclusion $E
\subset F$. Let $W \in E \subset \mbox{SU}(n)$.
It is a well-known result in linear algebra
that $W \in \mbox{SU}(n)$ can be decomposed as $W = M
\mbox{diag}(\lambda_1,\dots,\lambda_n) M^\dag$, where $M$ is
unitary, $\lambda_1, \dots, \lambda_n \in \mathbb{C}$, $| \lambda_i
| = 1$ and $\prod_{i=1}^n \lambda_i = 1$. Thus, $V(W) = \sum_{i=1}^{n} \Re(\lambda_i)$ and
$\mathrm{V}(W X^{\dag}_{\infty} B^j_k(0) X_{\infty}) =
\mathrm{V}(\mbox{diag}(\lambda_1,\dots,\lambda_n) (X_{\infty}
M)^\dag B^j_k(0) (X_{\infty} M)) = 0$, for each $1 \leq k \leq m$,
$j \in \mathbb{N}$. Since $X_{\infty} M$ is unitary, it is clear
that $N$:~$\mathfrak{su}(n) \rightarrow \mathfrak{su}(n)$ defined by
$N(Y)=(X_{\infty} M)^\dag Y (X_{\infty} M)$, for every $Y \in
\mathfrak{su}(n)$, is a linear surjective isomorphism. Now, by
assumption, system (\ref{cqs}) is regular and (\ref{rss}) is
satisfied. Hence, $\mathrm{V}(\mbox{diag}(\lambda_1,\dots,\lambda_n)
X ) = 0$, for every $X \in \mathfrak{su}(n)$, and thus
$\mathrm{V}(\mbox{diag}(\lambda_1,\dots,\lambda_n) D_\ell ) = 0$,
for each $1 \leq \ell \leq n$, where
$D_1=\mbox{diag}(\imath,-\imath,0,\dots,0)$,
$D_2=\mbox{diag}(0,\imath,-\imath,0,\dots,0),\dots,
D_{n-1}=\mbox{diag}(0,\dots,\imath,-\imath)$ and
$D_n=\mbox{diag}(\imath,0,\dots,0,-\imath)$ are the canonical
diagonal matrices of $\mathfrak{su}(n)$. From the diagonal structure of
$D_1, \dots, D_n$, we conclude that $\lambda_1, \dots, \lambda_n$
must satisfy $\Im(\lambda_1)=\dots=\Im(\lambda_n)$. This implies
that $W \in F$ and therefore $E \subset F$.
\end{proof}

The proof of Theorem~\ref{dct} is given below.

\begin{proof}
It is clear that $n = \max{(G)}$ because $I \in F$. We will first show that $G$ is finite. Let $x \in G$.
Then, there exist $\lambda_1, \dots, \lambda_n \in \mathbb{C}$ such
that $x = \sum_{j=1}^{n} \Re(\lambda_j)$ with (i) $\prod_{j=1}^{n}
\lambda_j$, (ii) $|\lambda_j|=1$ and (iii) $\Im(\lambda_1) = \dots = \Im(\lambda_n)$. Property (ii) implies that
$\lambda_j=e^{\imath \theta_j}$, for some $\theta_j \in \mathbb{R}$, and it follows from
(iii) that $\lambda_j=\lambda_1=e^{\imath \theta_1}$ or $\lambda_j=e^{\imath
(\pi - \theta_1)}$, for each $1 \leq j \leq n$. Let $n_1$ be the number of $j \in \{1,\dots,n\}$
such that $\lambda_j=\lambda_1$ and define $n_2 = n - n_1$.
Therefore, $x = n_1 \cos(\theta_1) + n_2 \cos(\pi - \theta_1) =
(n_1-n_2)\cos(\theta_1)$ with $n_1, (n_2 + 1) \in \{1, \dots, n\}$
and $n_1+n_2=n$. If $n_1=n_2$, then $x=0$. Thus, assume that $n_1
\neq n_2$. From property (i) we obtain that $e^{\imath n_1
\theta_1}e^{\imath n_2 (\pi - \theta_1)} = 1$. Hence, there exists
$k \in \mathbb{Z}$ such that $n_1 \theta_1 + n_2 (\pi - \theta_1) =
2 k \pi$. This relation implies that $\theta_1 = (2k-n_2)\pi / (n- 2
n_2)$. Note that $n_1, n_2, k$ depend on $x \in G$ and that $n_2$, $n_1 - n_2$ can only assume
a finite number of values. If we show that $\theta_1$ can only assume a finite number of values,
we will have shown that the same holds for $x \in G$, which implies that $G$ is finite.
It is clear that the function $\eta$:~$\mathbb{Z} \rightarrow
\mathbb{R}$ defined as $\eta(\ell)=\cos((2 \ell - n_2) \pi / (n - 2 n_2))$, for all
$\ell \in \mathbb{Z}$, has period $|n - 2 n_2| > 0$. Thus, the values assumed by $\theta_1$ must be finite in number.

Now, the convergence result will be shown. Recall that, for all $X \in
\mbox{SU}(n)$, we have that $-n \leq V(X) \leq n$ and that $V(X) =
n$ if and only if $X=I$. Let $\delta=\max(G \setminus \{n\})$.
Since $G$ is finite, we have that
\begin{equation}\label{ipp}
    \delta < x \leq n \Rightarrow x=n, \hspace{12pt} \mbox{ for all } x \in  G.
\end{equation}
Suppose that $V(W_{t_0})
> \delta$. Since $\mbox{SU}(n)$ is compact and $V$:~$M^n \rightarrow
\mathbb{R}$ is continuous, it follows that $V|\mbox{SU}(n)$ is
uniformly continuous. Define the function $\alpha$:~$\mathbb{R} \rightarrow
\mathbb{R}$ by $\alpha(t) = \mathrm{V}(W_q(t))$, for $t \in
\mathbb{R}$.
Recall that, by construction, we have that
$\dot{\mathrm{V}}(t,W)=\sum_{k=1}^{m} a_k(t,W)^2 \geq 0$, for all $(t,W)
\in \mathbb{R} \times M^n$. Note that $a_k$ and $\dot{\mathrm{V}}$
are smooth, for each $1 \leq k \leq m$. Since $\dot{\mathrm{V}}$ is
a non-negative function, we conclude that $\alpha$ is a smooth
non-decreasing function. Therefore, $V(W_q(t)) \geq V(W_{t_0}) >
\delta$, for all $t \geq t_0$. The uniform continuity of
$V|\mbox{SU}(n)$ then implies that there exists $\mu > 0$ such that
\begin{equation}\label{ucp}
    \| X - W_q(t) \| < \mu \Rightarrow V(X) > \delta, \hspace{12pt} \mbox{ for } t \geq t_0, \; X \in \mbox{SU}(n)
\end{equation}
(indeed, choose $\epsilon = V(W_{t_0}) - \delta > 0$). The
convergence result of Lemma~\ref{edcl-p} means that
\begin{equation*}
    \forall \epsilon > 0 \; \exists \overline{T} \in \mathbb{R} \; \forall t \geq \overline{T} \; \exists \alpha(t) \in F \mbox{ s.t. } \| \alpha(t) - W_q(t) \| < \epsilon.
\end{equation*}
Let $\epsilon > 0$ and define $\overline{\epsilon} = \min(\epsilon,
\mu)$. Thus,
\begin{equation*}
    \forall t \geq \overline{T} \; \exists \alpha(t) \in F \mbox{ s.t. } \| \alpha(t) - W_q(t) \| < \overline{\epsilon} \leq \mu,
\end{equation*}
for some $\overline{T} \in \mathbb{R}$. Define
$\widetilde{T}=\max(\overline{T},t_0)$ and let $t \geq
\widetilde{T}$. Since $\alpha(t) \in F \subset \mbox{SU}(n)$ and
$V(F) \subset G$, (\ref{ucp}) gives that $\delta < V(\alpha(t)) \in
G$. However, $\delta < V(\alpha(t)) \leq n$. Therefore, from
(\ref{ipp}), we obtain that $V(\alpha(t)) = n$, which implies that
$\alpha(t)=I$. We have thus shown that $\lim_{t \rightarrow \infty}
W_q(t) = I$.
\end{proof}

\def\cprime{$'$}

\end{document}